\documentclass[a4paper,11pt]{amsart}
\usepackage{amsmath,amssymb,latexsym,amsfonts}
\usepackage[babel=true]{csquotes} 

\newtheorem{theorem}{Theorem}[section]
\newtheorem{lemma}[theorem]{Lemma}
\newtheorem{prop}[theorem]{Proposition}
\newtheorem{corollary}[theorem]{Corollary}

\theoremstyle{definition}		
\newtheorem{definition}[theorem]{Definition}
\newtheorem{examp}[theorem]{Example}

\newtheorem{rmk}[theorem]{Remark}

\newcommand{\N}{\mathbb{N}}
\newcommand{\R}{\mathbb{R}}

\newcommand{\C}{\mathbb{C}}
\newcommand{\Cc}{\mathcal{C}}
\newcommand{\Hc}{\mathcal{H}}

\newcommand{\f}{\varphi}
\newcommand{\p}{\psi}
\newcommand{\Fc}{\mathcal{F}}

\newcommand{\St}{\widetilde{S}}

\newcommand{\Gt}{\widetilde{\Gamma}}
\newcommand{\Pm}{\mathcal{P}_m}

\newcommand{\pO}{\partial \Omega}

\newcommand{\loc}{L^{\infty}_{loc}}

\newcommand{\Lct}{\mathcal{L}^*}

\usepackage{hyperref}
\usepackage{hyperref}
\usepackage{xcolor}
\usepackage{comment}
\usepackage{dsfont}
\usepackage{yhmath}
\definecolor{violet}{rgb}{0.0,0.2,0.7}
\definecolor{rouge}{cmyk}{0.0,0.6,0.4,0.3}
\definecolor{rouge2}{rgb}{0.8,0.0,0.2}
\hypersetup{
    bookmarks=true,         
    unicode=false,          
    pdftoolbar=true,        
    pdfmenubar=true,        
    pdffitwindow=false,     
    pdfstartview={FitH},    
    pdftitle={},    
    pdfauthor={},     
    colorlinks=true,       
   linkcolor=blue,          
    citecolor=blue,        
    filecolor=black,      
    urlcolor=cyan}           

\begin{document}
\title{Viscosity solutions to complex Hessian equations}
\author{Lu Hoang Chinh}
\date{\today}
\maketitle
\begin{abstract}
We study viscosity solutions to complex Hessian equations. 
In the local case, we consider $\Omega$ a bounded domain
 in $\C^n,$ $\beta$  the standard K\"{a}hler form in $\C^n$ 
 and  $1\leq m\leq n.$  Under some suitable conditions on $F, g$, 
 we prove that the equation $(dd^c \f)^m\wedge\beta^{n-m}=F(x,\f)\beta^n,\ \f=g$ 
 on $\pO$ admits  a unique  viscosity solution modulo the existence of subsolution 
 and supersolution. If moreover, the datum is H\"{o}lder continuous then so is the solution.  
 In the global case, let $(X,\omega)$ be a compact Hermitian homogeneous 
 manifold where $\omega$ is an invariant Hermitian metric (not necessarily  K\"{a}hler). 
 We prove that the  equation $(\omega+dd^c\f)^m\wedge\omega^{n-m}=F(x,\f)\omega^n$ 
 has a unique viscosity solution under some natural conditions on $F.$ 
 \end{abstract}
\section{Introduction}
The complex Hessian equation has been studied intensively in recent years.  
In \cite{Li04}, Li solved Dirichlet problems for  complex Hessian equations 
in $m$-pseudoconvex domains with smooth right-hand side and smooth boundary data by using the continuity method. 
In \cite{Bl05}, Blocki considered degenerate complex Hessian equations in $\C^n$  and developed a first step of a potential theory for this equation. Recently, Sadullaev and Abdullaev also studied capacities and polar sets for $m$-subharmonic functions \cite{SA12}.  

Hou \cite{Hou09}, Jbilou \cite{Jb10} and Kokarev \cite{Kok10} began the program of solving  the  non-degenerate complex Hessian equation on compact K\"{a}hler manifolds  four years ago. This is a generalization of the famous Calabi-Yau equation \cite{Y78}. In \cite{Kok10}, this equation is solved under rather restrictive assumptions on the underlined manifold.  Hou \cite{Hou09} and Jbilou \cite{Jb10}  independently solved this equation in the case  the manifold has non-negative holomorphic bisectional curvatures. The curvature assumption served as a technical point in an a priori $\Cc^2$ estimates and people wanted to remove it. Later on, Hou, Ma and Wu \cite{HMW10} provided an important $\Cc^2$ estimate without this hypothesis.  
 Using this estimate and a blowing-up analysis, Dinew and Kolodziej recently 
 solved  the equation in full generality \cite{DK12}. 

Degenerate complex Hessian equations on compact K\"{a}hler manifold   
were considered in \cite{DK11} and \cite{Chi12}.  This approach is global 
in nature since it relies on some difficult integrations by parts. 

The study of real Hessian equations is a classical subject which has been
 developed  previously in many papers, for example
  \cite{CNS85,CW01,ITW04,Kry95,Lab02,Tr95,TW99,Ur01,W09}.

The viscosity method introduced in \cite{Lio83} (see also \cite{CIL92} for a survey)
 is purely local and very efficient to study  weak solutions to nonlinear elliptic partial 
 differential equations.  In \cite{EGZ11} the authors used this method to study degenerate 
 complex Monge-Amp\`ere equation on compact K\"{a}hler manifolds. In the local
 context, using this approach Wang \cite{YW10} considered the Dirichlet
 problem for complex Monge-Amp\`ere  equations where the right-hand side 
 also depends on the solution. 
\medskip

From recent developments on viscosity method applied to complex
 Monge-Amp\`ere equations  it is natural to develop such a treatment 
 for complex Hessian equations. It is the main purpose of this paper, 
 precisely we consider the following complex Hessian equation:
\begin{equation}\label{eq: hessian}
-(dd^c \f)^m\wedge \beta^{n-m}+F(x,\f)\beta^n=0, 
\end{equation}
with boundary value $\f=g$ on $\pO,$ where $1\leq m\leq n$, $\Omega$ 
is a bounded domain in $\C^n,$ $\beta$ is the standard K\"{a}hler form in $\C^n$, 

\begin{equation}\label{eq: V1}
g\in \Cc(\pO),\ \text{and}
\end{equation} 
\begin{eqnarray}\label{eq: V2}
&&F(x,t) \ \text{is a continuous function on}\  \Omega\times \R \rightarrow \R^+\\
&& \text{which is non-decreasing in the second variable}.\nonumber
\end{eqnarray}

We say that $F^{1/m}$ is $\gamma$-H\"{o}lder continuous  uniformly in $t$ if 

\begin{eqnarray}\label{eq: V3}
\sup_{\vert t\vert \leq M} \sup_{x\neq y\in\bar{\Omega}} 
\frac{\vert F^{1/m}(x,t)-F^{1/m}(y,t)\vert}{\vert x-y\vert^{\gamma}}<+\infty, \ \forall M>0.
\end{eqnarray}

Equation (\ref{eq: hessian})  is nonlinear degenerate second order elliptic in the viscosity 
sense (see \cite{CIL92}) when restricted to $m$-subharmonic functions. So, we  can use the concepts of subsolutions and supersolutions.

\medskip

The main results are the following:

\medskip

\noindent {\bf Theorem A.} {\it Let $g, F$ be functions satisfying (\ref{eq: V1}) and (\ref{eq: V2}) 
respectively. Assume that there exist a bounded subsolution $u$ and a bounded supersolution $v$ 
to (\ref{eq: hessian}) such that $u_*=v^*=g$ on $\pO.$ Then there exists a unique viscosity solution
 to (\ref{eq: hessian}) with boundary value $g$.  It is also the unique potential solution.}
\medskip

\noindent {\bf Theorem B.} {\it With the same assumption as in Theorem A, assume moreover that 
$u,v$ are $\gamma$-H\"{o}lder continuous in $\bar{\Omega}$ and $F$ satisfies (\ref{eq: V3}). 
Then the unique solution of (\ref{eq: hessian}) with boundary value $g$ is also $\gamma$-H\"{o}lder 
continuous in $\bar{\Omega}.$}
\medskip

In Theorem A and Theorem B, to solve the equation we need to find a subsolution and a supersolution.  
When the domain $\Omega$ is strictly pseudoconvex these existences are guaranteed. 
\medskip

\noindent {\bf Corollary C.} {\it Assume that $g\in \Cc(\pO)$ and $F$ satisfies (\ref{eq: V2}) and 
 (\ref{eq: V3}). If $\Omega$ is strongly pseudoconvex then (\ref{eq: hessian}) has a unique viscosity 
 solution with boundary value $g.$ If, moreover,  $g$ is $(2\gamma)$-H\"{o}lder continuous and $F$ 
 satisfies (\ref{eq: V3})  with $0<\gamma\leq 1$ then the unique solution is  $\gamma$-H\"{o}lder continuous.}
\medskip

\noindent {\bf Remark.} When $m=n$ we recover the results in \cite{YW10}.
\medskip

We also study viscosity solutions on compact homogeneous Hermitian manifolds. We assume that $X$ 
 is a Hermitian manifold with a Hermitian metric $\omega$ such that the following conditions are verified:

\noindent (H1) $X=G/H$ where $G$ is a connected  Lie group and $H$ is a closed subgroup. 

\noindent (H2)  There exists a compact subgroup $K\subset G$ which acts transitively on $X.$

\noindent (H3) $\omega$  is invariant by $K.$

\noindent{\bf Theorem D.} {\it Assume that $(X,\omega)$ satisfies (H1), (H2) and (H3).  
Let  $F(x,t)$ be a continuous function which is increasing in $t$ and  assume that there 
exist $t_0, t_1\in \R$ such that 
\begin{equation}\label{eq: V4}
F(x,t_0)\leq 1\leq F(x,t_1),\ \forall x\in X.
\end{equation}
Then there exists a unique viscosity solution to 
$$
-(\omega+dd^c\f)^m\wedge \omega^{n-m}+F(x,\f)\omega^n=0.
$$}
\medskip

\noindent{\bf Remark.} Our proof here does not use a priori $\Cc^2$ estimates in 
contrast with  a similar  result in \cite{Chi12} where we use  potential method \cite{Chi12} 
and existence result of \cite{DK12} which relies on $\Cc^2$ estimate of Hou, Ma and Wu \cite{HMW10}.
Moreover, we do not assume that $\omega$ is closed.  An example  of compact Hermitian manifold 
satisfying (H1), (H2), (H3) which is not K\"{a}hler was given to us by  Karl Oeljeklaus to 
 whom we are indebted (Example \ref{example: Karl}).
\medskip

\section{Preliminaries}
In this section, $\Omega$ is a bounded domain and $\beta$ is the 
standard K\"{a}hler form in $\C^n.$ 
We introduce the notion  and basic properties of $m$-subharmonic 
functions in the local context and 
one of $(\omega,m)$-subharmonic functions on compact K\"{a}hler 
manifolds.
\subsection{Elementary symmetric functions}
We begin by a brief review of elementary symmetric functions
 (see \cite{Bl05}, \cite{CW01}, \cite{Ga59}). 
We use the notations in \cite{Bl05}.\\
Let $1\leq k\leq n$ be natural numbers. The elementary symmetric 
function of order $k$ is defined by 
\begin{equation*}
S_k(\lambda)= \sum_{1\leq i_1<i_2<...<i_k\leq
n}\lambda_{i_1}\lambda_{i_2}...\lambda_{i_k},\ \ \lambda=(\lambda_1,...,\lambda_n)\in \R^n.
\end{equation*}
Let $\Gamma_k$ denote the closure of the connected component of $\{
S_k(\lambda)>0\}$ containing $(1,...,1).$ It is easy to show that
$$
\Gamma_k=\big\{\lambda\in \R^n \ / \  S_k(\lambda_1+t,...,\lambda_n+t)\geq 0,  \ \forall t\geq 0\big\}.
$$
and hence
$$
\Gamma_k= \big\{\lambda\in \R^n\ / \ S_j(\lambda)\geq 0,\ \ \forall 1\leq j\leq k\big\}.
$$
We have an obvious inclusion $\Gamma_n\subset ...\subset \Gamma_1.$\\
The set $\Gamma_k$ is a convex cone in $\R^n$ and $S_k^{1/k}$ is concave on $\Gamma_k$ \cite{Ga59}. 

Let $\Hc$ denote the vector space (over $\R$) of complex Hermitian matrices of dimension $n\times n.$ For $A\in \Hc$ we set
$$
\St_k(A)=S_k(\lambda(A)),
$$
where $\lambda(A)\in \R^n$  is the vector of eigenvalues of $A.$ The function $\St_k$ can also be defined as the sum of all principal minors of order $k$,
$$
\St_k(A)=\sum_{\vert I\vert =k}A_{II}.
$$
From the latter we see that $\St_k$ is a homogeneous polynomial of order $k$ on $\Hc$ which is hyperbolic with respect to the identity matrix $I$ (that is for every $A\in
\St$ the equation $\St_k(A+tI)=0$ has $n$ real roots; see \cite{Ga59}). As in
\cite{Ga59} (see also \cite{Bl05}), the cone
$$
\Gt_k:=\big\{A\in \Hc\ / \ \St_k(A+tI)\geq 0, \forall t\geq 0\big\}=\{A\in \Hc\ / \ \lambda(A)\in \Gamma_k\}
$$
is convex and the function $\St_k^{1/k}$ is concave on $\Gt_k.$
\subsection{m-subharmonic functions and the Hessian operator}
We associate real (1,1)-forms $\alpha$ in $\C^n$ with Hermitian matrices $[a_{j\bar{k}}]$ by
$$
\alpha=\frac{i}{\pi}\sum_{j,k}a_{j\bar{k}}dz_j\wedge d\bar{z_k}.
$$
Then the canonical K\"{a}hler form $\beta$ is associated with the identity matrix I. It is easy to see that 	 
\begin{equation*}\label{eq: form to matrix}
\binom{n}{k}\alpha^k\wedge \beta^{n-k}=\St_k(A)\beta^n.
\end{equation*}

\begin{definition}
Let $\alpha$ be a real $(1,1)$-form on $\Omega$. We say that $\alpha$ is $m$-positive at a given point $P\in \Omega$ if at this point we have 
$$
\alpha^j\wedge\beta^{n-j}\geq 0, \ \ \forall j=1,...,m.
$$
$\alpha$ is called $m$-positive if it is $m$-positive at any point of $\Omega.$ If there is no confusion we also denote by $\Gt_m$ the set of $m$-positive (1,1)-forms. \\
Let $T$ be a current of bidegree $(n-k,n-k)$ with $k\leq m$. Then $T$ is called $m$-positive if 
$$
\alpha_1\wedge...\wedge \alpha_k\wedge T\geq 0,
$$
for all $m$-positive $(1,1)$-forms $\alpha_1,...,\alpha_k.$
\end{definition}

 \begin{definition}
A function $u: \Omega\rightarrow \R\cup \{-\infty\}$ is called $m$-subharmonic if it is subharmonic and 
$$
dd^c u\wedge \alpha_1\wedge...\wedge\alpha_{m-1}\wedge\beta^{n-m}\geq 0,
$$
for every $m$-positive (1,1)-forms $\alpha_1,...,\alpha_{m-1}.$
The class of all $m$-subharmonic functions in $\Omega$ will be denoted by $\Pm(\Omega).$
 \end{definition}

We summarize basic properties of $m$-subharmonic functions in the following: 

\begin{prop}\cite{Bl05}\label{prop: basic property of m subharmonic functions}
(i) If $u$ is $\mathcal{C}^2$ smooth then $u$ is $m$-subharmonic if and only if the form $dd^cu$ is $m$-positive at every point in $\Omega.$

(ii) If $u,v \in \Pm(\Omega)$ then $\lambda u+\mu v\in \Pm(\Omega), \forall \lambda, \mu >0.$

(iii) If $u$ is $m$-subharmonic in $\Omega$ then the standard regularization $u\star \chi_{\epsilon}$ is also $m$-subharmonic in $\Omega_{\epsilon}:=\{x\in \Omega: d(x,\partial \Omega)>\epsilon\}$.

(iv) If $(u_l)\subset \Pm(\Omega)$ is  locally uniformly bounded from above then $(\sup u_l)^{\star}\in \Pm(\Omega)$, where $v^{\star}$ is the upper semicontinuous regularization of $v$. 

(v) $PSH=\mathcal{P}_n\subset ...\subset \mathcal{P}_1=SH.$

(vi)  Let $\emptyset \neq U\subset \Omega$ be a proper open subset such that $\partial U\cap \Omega$ is relatively compact in $\Omega$. If $u\in \mathcal{P}_m(\Omega)$, $v\in \Pm(U)$ and $\limsup_{x\to y}v(x)\leq u(y)$ for each $y\in \partial U\cap \Omega$ then the function $w$, defined by
$$
 w= \left\{ {\begin{array}{*{20}c} {u\ \textrm{on}\ \Omega \setminus U}\\{\max (u,v) \ \textrm{on} \ U }\end{array}}\right., 
 $$
is $m$-subharmonic in $\Omega.$
\end{prop}

For locally bounded $m$-subharmonic functions $u_1,...,u_p$ ($p\leq m$) and a closed $m$-positive current $T$ we can inductively define a closed $m$-positive current $dd^c u_1\wedge...\wedge dd^c u_p\wedge T$ (following Bedford and Taylor \cite{BT76}).
\begin{lemma}\label{lem: symmetric}
Let $u_1,...,u_k (k\leq m)$ be locally bounded $m$-subharmonic functions in $\Omega$ and let $T$ be a closed $m$-positive current of bidegree $(n-p,n-p)$ ($p\geq k$). Then we can define inductively a closed $m$-positive current 
$$
dd^cu_1\wedge dd^cu_2\wedge ...\wedge dd^cu_k\wedge T,
$$
and the product is symmetric, i.e. 
$$
dd^cu_1\wedge dd^cu_2\wedge ...\wedge dd^cu_p\wedge T=
dd^cu_{\sigma(1)}\wedge dd^cu_{\sigma(2)}\wedge ...\wedge dd^cu_{\sigma(k)}\wedge T,
$$
for every permutation $\sigma: \{1,...,k\}\to \{1,...,k\}.$

In particular, the Hessian measure of $u\in \Pm(\Omega)\cap\loc$ is well defined as
$$
H_m(u)=(dd^cu)^m\wedge\beta^{n-m}.
$$
\end{lemma}
\subsection{$(\omega,m)$-subharmonic functions}
In this section,  $(X,\omega)$ is a compact K\"{a}hler manifold and  $U\subset X$  is an open subset  contained in a local chart. 
\begin{definition}\label{def: weakly subharmonic}
A function $u\in L^1(U)$ is called weakly $\omega$-subharmonic if 
$$
dd^cu\wedge \omega^{n-1}\geq 0,
$$
in the weak sense of currents.
\end{definition}
\noindent Thanks to Littman \cite{Lit63} we have the following approximation properties.
\begin{prop}\label{prop: Littman}
Let $u$ be a weakly $\omega$-subharmonic function in $U$. Then there exists a one parameter family of functions $u_h$ with the following properties: For every  compact subset $U'\subset U$

a)  $u_h$ is smooth in $U'$ for $h$ sufficiently large,

b)  $dd^cu_h\wedge \omega^{n-1}\geq 0$ in $U',$

c) $u_h$ is non-increasing with increasing $h,$ and $\lim_{h\to \infty}u_h(x)=u(x)$ almost everywhere  in $U',$

d) $u_h$ is given explicitly as $u_h(y)=\int_{U} K_h(x,y)u(x)dx,$ where $K_h$ is a smooth non-negative function and $\int_{U} K_h(x,y)dy \to 1,$ uniformly in $x\in U'.$ 
\end{prop}

\begin{definition}
A function $u$ is called $\omega$-subharmonic if it is weakly $\omega$-subharmonic  and for every $U'\Subset U$,  $\lim_{h\to\infty}u_h(x)=u(x), \forall x\in U',$ where $u_h$ is constructed as in Proposition \ref{prop: Littman}. 
\end{definition}
\begin{rmk}
Any continuous weakly $\omega$-subharmonic function is $\omega$-subharmonic. 

If $(u_j)$ is a sequence of continuous $\omega$-subharmonic functions decreasing to $u\not \equiv -\infty$ then $u$ is $\omega$-subharmonic. 

If $u$ is weakly $\omega$-subharmonic then the pointwise limit of $(u_h)$ is an $\omega$-subharmonic function.

Let $(u_j)$ be a sequence of $\omega$-subharmonic functions and $(u_j)$ is uniformly bounded from above. Then $u:=(\limsup_j u_j)^{\star}$ is $\omega$-subharmonic, where for a function $v$, $v^{\star}$ denotes the upper semicontinuous regularization of $v.$
\end{rmk}

\begin{definition}
Let $\alpha$ be a real $(1,1)$-form on $X$. We say that $\alpha$ is
$(\omega,m)$-positive at a given point $P\in X$ if at this point we have
$$
\alpha^k\wedge\omega^{n-k}\geq 0, \ \ \forall k=1,...,m.
$$
We say that $\alpha$ is $(\omega,m)$-positive if it is $(\omega,m)$-positive at
any point of $X.$
\end{definition}
\begin{rmk}
Locally at $P\in X$ with local coordinates $z_1,...,z_n$, we have
$$
\alpha=\frac{i}{\pi}\sum_{j,k}\alpha_{j\bar{k}}dz_j\wedge d\bar{z_k},
$$
and 
$$
\omega=\frac{i}{\pi}\sum_{j,k}g_{j\bar{k}}dz_j\wedge d \bar{z_k}.
$$
Then $\alpha$ is $(\omega,m)$-positive at $P$ if and only if the vector of eigenvalues
$\lambda(g^{-1}\alpha)=(\lambda_1,...,\lambda_n)$ of the matrix
$\alpha_{j\bar{k}}(P)$ with respect to the matrix $g_{j\bar{k}}(P)$ is in
$\Gamma_m$. These eigenvalues are independent of any choice of local
coordinates. 
\end{rmk}
\noindent Following Blocki \cite{Bl05} we can define $(\omega,m)$-subharmonicity for  (non-smooth) functions.

\begin{definition}\label{def: m sub non smooth}
A function $\varphi: X\rightarrow \R \cup \{-\infty\}$ is called $(\omega,m)$-subharmonic if the following conditions hold:

(i)  in any local chart $\Omega,$ given $\rho$ a local potential of $\omega$ and set $u:=\rho+\varphi$, then $u$ is $\omega$-subharmonic,  

(ii) for every smooth $(\omega,m)$-positive forms
$\beta_1,...,\beta_{m-1}$ we have, in the weak sense of distributions,
$$
(\omega+dd^c\varphi)\wedge \beta_1\wedge...\wedge \beta_{m-1}\wedge
\omega^{n-m}\geq 0.
$$
\end{definition}

\noindent Let  $SH_m(X,\omega)$ be the set of all $(\omega,m)$-subharmonic functions on $X.$ Observe  that, by definition, any $\varphi\in SH_m(X,\omega)$ is upper semicontinuous. 

\noindent The following properties of $(\omega,m)$-subharmonic functions are easy to show.

\begin{prop}\label{prop: properties of m sub functions}
(i)  If $\varphi\in \Cc^2(X)$ then $\varphi$ is $(\omega,m)$-subharmonic if the form $(\omega+dd^c\varphi)$ is $(\omega,m)$-positive, or equivalently 
$$
(\omega+dd^c\f)\wedge(\omega+dd^cu_1)\wedge...\wedge(\omega+dd^c u_{m-1})\wedge\omega^{n-m}\geq 0,
$$
for all $\Cc^2$ $(\omega,m)$-subharmonic functions $u_1,...,u_{m-1}.$

(ii) If $\varphi,\psi\in SH_m(X,\omega)$ then $\max (\varphi,\psi)\in SH_m(X,\omega).$

(iii) If $\varphi,\psi\in SH_m(X,\omega)$ and $\lambda\in [0,1]$ then $\lambda \varphi+(1-\lambda)\psi\in SH_m(X,\omega).$

(iv) If $(\varphi_j)\subset SH_m(X,\omega)$ is uniformly bounded from above then  
$$
(\limsup_{j}\varphi_j)^{\star}\in SH_m(X,\omega).
$$
\end{prop}

\section{Viscosity solutions vs. potential solutions}
In this section we introduce the notion of viscosity (sub, super)-solutions to degenerate complex Hessian equations and systematically compare them with potential ones. We prove an important comparison principle which is the key point in the proof of our main results. The idea of our proof is taken from \cite{EGZ11}, \cite{YW10, CC95},  \cite{CIL92}.  
\begin{definition}
Let $u: \Omega\rightarrow \R\cup\{-\infty\}$ be a function. Let $\f$ be a $\Cc^2$ function in a neighborhood of $x_0\in \Omega.$ We say that $\f$ touches $u$ from above (resp. below) at $x_0$ if $\f(x_0)=u(x_0)$ and $\f(x)\geq u(x)$ (resp. $\f(x)\leq u(x)$) for every $x$ in a neighborhood of $x_0.$
\end{definition}
\begin{definition}\label{def: viscosity subsolution} 
An upper semicontinuous function $\f: \Omega\rightarrow \R\cup \{-\infty\}$ is  a viscosity subsolution to
\begin{equation}\label{eq: heq 1}
-(dd^c \f)^m\wedge \beta^{n-m}+F(x,\f)\beta^n=0
\end{equation}
if $\f\not\equiv -\infty$ and for any $x_0\in \Omega$ and any $\Cc^2$ function $q$ which touches $\f$ from above at $x_0$ then 
$$
H_m(q)\geq F(x,q)\beta^n, \ \text{at}\ x_0.
$$
Here we use the notation $H_m(u)=(dd^cu)^m\wedge \beta^{n-m}$ for $u\in \Cc^2(X).$ We also say that $H_m(\f)\geq F(x,q)\beta^n$ \enquote{in the viscosity sense}.
\end{definition}
\begin{definition}\label{def: viscosity supersolution}
A lower semicontinuous function $\f: X\rightarrow \R\cup \{+\infty\}$ is  a viscosity supersolution to (\ref{eq: heq 1}) if $\f\not\equiv +\infty$ and for any $x_0\in X$ and any $\Cc^2$ function $q$ which touches $\f$ from below at $x_0$ then
$$
[(dd^cq)^m\wedge \beta^{n-m}]_+\leq F(x,q)\beta^n, \ \text{at} \ x_0.
$$ 
Here $[\alpha^m\wedge\beta^{n-m}]_+$ is defined to be itself  if $\alpha$ is $m$-positive and $0$ otherwise.
\end{definition}

\begin{rmk} If $u\in \Cc^2(\Omega)$ then $H_m(\f)\geq  F(x,\f)\beta^n$ (or $[H_m(\f)]_+\leq F(x,\f)\beta^n$) holds in the viscosity sense iff it holds in the usual sense.
\end{rmk}

\begin{definition}
A function $\f: X\rightarrow \R$ is a viscosity solution to (\ref{eq: hessian}) if it is both a subsolution and a supersolution. Thus, a viscosity solution is automatically continuous. 
\end{definition}
The notion of viscosity subsolutions is stable under taking maximum. It is also stable along monotone sequences as the following lemma shows.  
\begin{lemma}\label{lem: decreasing sequence of viscosity subsolution}
Assume that $F: \Omega\times \R\rightarrow \R^+$ is a continuous function. Let $(\f_j)$ be a monotone sequence of viscosity subsolutions of equation
\begin{equation}\label{eq: limsup}
-(dd^c u)^m\wedge \beta^{n-m} + F(x,u)\beta^n=0,
\end{equation}
If $\f_j$ is uniformly bounded from above  and $\f:=(\lim\f_j)^*\not \equiv -\infty$ then $\f$ is also a viscosity subsolution of (\ref{eq: limsup}).
\end{lemma}

\begin{proof} The proof can be found in \cite{CIL92}. For convenience, we reproduce it here.
Observe that if $z_j\to z$ then  
$$
\limsup_{j\to+\infty} \f_j(z_j)\leq \f(z).
$$
Fix $x_0\in \Omega$ and $q$ a $\Cc^2$ function in a neighborhood of $x_0$, say $B(x_0,r)\subset \Omega$ which touches $\f$ from above at $x_0$.  We can choose a sequence  $(x_j)\subset B=\bar{B}(x_0,r/2)$ converging to $x_0$ and  a subsequence of $(\f_j)$ (still denoted by $(\f_j)$) such that 
$$
\f_j(x_j)\to \f(x_0).
$$ 
Fix $\epsilon>0.$ For each $j,$ let $y_j$ be the maximum point of $\f_j-q-\epsilon\vert x-x_0\vert^2$ on $B.$ Then
\begin{eqnarray}\label{eq: viscosity solution 1}
\f_j(x_j)-q(x_j)-\epsilon\vert x_j-x_0\vert^2\leq \f_j(y_j)-q(y_j)-\epsilon\vert y_j-x_0\vert^2.
\end{eqnarray}
We claim that $y_j\to x_0.$ Indeed, assume  that $y_j\to y\in B.$ Letting $j\to +\infty$ in (\ref{eq: viscosity solution 1}) and noting that $\limsup\f_j(y_j)\leq \f(y)$, we get 
$$
0\leq \f(y)-q(y)-\epsilon\vert y-x_0\vert^2.
$$ 
Remember that $q$ touches $\f$ above in $B$ at $x_0$ and $y\in B.$ Thus, the above inequality implies that $y=x_0,$ which means $y_j\to x_0.$ Then again by (\ref{eq: viscosity solution 1}) we deduce that $\f_j(y_j)\to \f(x_0).$

For $j$ large enough, the function 
$$
q+\epsilon\vert x-x_0\vert^2+\f_j(y_j)-q(y_j)-\epsilon\vert y_j-x_0\vert^2 
$$
 touches $\f_j$ from above at $y_j.$ Thus
 $$
 H_m(q+\epsilon \vert x-x_0\vert^2)(y_j)\geq F(y_j,\f_j(y_j))\beta^n.
 $$
It suffices now to let $j\to+\infty$.
\end{proof}

When $F\equiv 0,$ viscosity subsolutions of (\ref{eq: hessian}) are exactly $m$-subharmonic functions.
\begin{lemma}\label{lem: viscosity sense and usual sense}  A function $u$ is $m$-subharmonic in $\Omega$  if and only if it is a viscosity subsolution of
\begin{equation}\label{eq: potential vs viscosity 0 rhs}
-(dd^c u)^m\wedge \beta^{n-m}= 0.
\end{equation}
\end{lemma}
\begin{proof} Assume  that $u$ is  $m$-subharmonic in $\Omega$ and let $u_{\epsilon}$ be its standard smooth regularization. Then $u_{\epsilon}$ is $m$-subharmonic and smooth, hence $u_{\epsilon}$ is a classical subsolution of (\ref{eq: potential vs viscosity 0 rhs}). Thus,  it follows from Lemma \ref{lem: decreasing sequence of viscosity subsolution} that $u$ is a viscosity subsolution of  (\ref{eq: potential vs viscosity 0 rhs}).

Conversely, assume that $u$ is a viscosity subsolution of (\ref{eq: potential vs viscosity 0 rhs}). Fix 
$\alpha_1,...,\alpha_{m-1}$ $m$-positive (1,1)-forms with constant coefficients such that 
$$
\alpha_1\wedge...\wedge \alpha_{m-1}\wedge \beta^{n-m}
$$
is strictly positive. Let $x_0\in \Omega$ and $q\in \mathcal{C}^2(V_{x_0})$ such that $u-q$ has a local maximum at $x_0.$ Then for any $\epsilon>0$, $q+\epsilon \vert z-z_0\vert^2$ also touches $u$ from above. By the definition of viscosity subsolutions, we have 
$$
(dd^c q +\epsilon \beta)^m\wedge \beta^{n-m}\geq 0, \forall \epsilon>0,
$$
which means that  the Hessian matrix $\dfrac{\partial^2q }{\partial z_j\partial \bar{z}_k}(x_0)$ is $m$-positive. Hence 
$$
L_{\alpha} q:=dd^c q \wedge \alpha_1\wedge...\wedge \alpha_{m-1}\wedge \beta^{n-m} \geq 0,
$$ 
holds at $x_0.$ 

This implies $L_{\alpha} u\geq 0$ in the viscosity sense. In appropriate complex coordinates this constant coefficient differential operator is the Laplace operator. Hence, \cite{Hor94} Proposition 3.2.10' p. 147 implies that $u$ is $L_{\alpha}$-subharmonic hence is $L^1_{\rm loc}(V_{x_0})$ and satisfies $L_{\alpha} u\geq 0$ in the sense of distributions. Since $\alpha_1,...,\alpha_{m-1}$ were taken arbitrarily, by continuity we have 
$$
dd^c u\wedge \alpha_1\wedge...\wedge \alpha_{m-1}\wedge \beta^{n-m}\geq 0 
$$
in the sense of distributions for any $m$-positive (1,1)-forms $\alpha.$ Therefore, $u$ is $m$-subharmonic.
\end{proof}
\begin{corollary}
Lemma \ref{lem: decreasing sequence of viscosity subsolution} still holds if the sequence $\f_j$ is not monotone.
\end{corollary}
\begin{proof}
For each $j,$ set
$$
u_j=(\sup_{k\geq j} \f_k)^*,  \ \ v_l:=\max(\f_j,...,\f_{j+l}).
$$
Since the notion of viscosity subsolution is stable under taking the maximum, we deduce that $v_l$ is a viscosity subsolution of (\ref{eq: limsup}). Observe that $u_j=(\sup_{l\geq 0} v_l)^*$ and the sequence $(v_l)$ is monotone. It follows from what we have done before that $u_j$ is a viscosity subsolution of (\ref{eq: limsup}).  By Lemma  \ref{lem: viscosity sense and usual sense}  each $\f_j$ is $m$-subharmonic. Hence, $u_j\downarrow \f$ and the proof is complete.
\end{proof} 
For real $(1,1)$-form $\alpha$, we denote by
$$
S_m(\alpha):=\frac{\alpha^m\wedge \beta^{n-m}}{\beta^n}.
$$
Set 
$$
U_m:=\{\alpha\in \Gt_m\ \text{of constant coefficients such that}\ S_m(\alpha)=1\}.
$$
It is elementary to prove the following lemma: 
\begin{lemma}\label{lem: Gav}
Let $\alpha$ be a real $m$-positive $(1,1)$-form. Then the following identity holds 
$$
(S_m(\alpha))^{1/m}=\inf\Big\{\frac{\alpha\wedge\alpha_1\wedge...\wedge\alpha_{m-1}\wedge \beta^{n-m}}{\beta^n} \ / \ \alpha_j\in U_m, \forall j\Big\}.
$$
\end{lemma}

Now, we compare viscosity and potential subsolutions when the right-hand side $F(x,t)$ does not depend on  $t.$
\begin{prop}\label{prop: viscosity subsolution vs potential subsolution continuous rhs}
Let $\f$ be a bounded upper semicontinuous function in $\Omega$ and $0\leq f$ be a continuous function. 

(i) If $\f$ is $m$-subharmonic such that 
\begin{equation}\label{eq: viscosity vs potential 1}
H_m(\f)\geq f\beta^n
\end{equation} 
in the potential sense then it also holds in the viscosity sense. 

(ii) Conversely, if (\ref{eq: viscosity vs potential 1}) holds in the viscosity sense then $\f$ is $m$-subharmonic and the inequality holds in the potential sense.
\end{prop}

\begin{proof}
We follow  \cite{EGZ11}. 

\noindent {\bf Proof of (i):} Let $x_0\in \Omega$ and assume that $q$ is a $\Cc^2$ functions which touches $\f$ from above at $x_0.$ Suppose that $H_m(q(x_0))<f(x_0)\beta^n.$ There exists $\epsilon>0$ such that $H_m(q_{\epsilon})<f\beta^n$ in a neighborhood of $x_0$ since $f$ is continuous, here $q_{\epsilon}=q+\epsilon \vert z-x_0\vert^2.$ It follows from the proof of Lemma \ref{lem: viscosity sense and usual sense} that $q_{\epsilon}$ is $m$-subharmonic in a neighborhood of $x_0$, say $B.$ Now, for $\delta>0$ small enough, we have $q_{\epsilon}-\delta \geq \f$ on $\partial B$ but it fails at $x_0$ which contradicts the potential comparison principle (see Theorem 1.14 and Corollary 1.15 in \cite{Cuo12}). 

\noindent {\bf Proof of (ii):} We proceed steps by steps.

\noindent {\bf Step 1: Assume that $0<f$ is  smooth.} Let $x_0\in \Omega$ and assume that $q$ is a $\Cc^2$ functions which touches $\f$ from above at $x_0.$ Fix $\alpha_1,...,\alpha_{m-1}\in U_m.$

We can find $h\in \Cc^2(\{x_0\})$ such that $L_{\alpha} h=f^{1/m}\beta^n.$ As in the proof of Lemma \ref{lem: viscosity sense and usual sense}, we can prove that $\f-h$ is $L_{\alpha}$-subharmonic, which gives $L_{\alpha} \f \geq L_{\alpha} h=f^{1/m}\beta^n$ in the potential sense.   

Consider the standard regularization $\f_{\epsilon}$ of $\f$ by convolution with a smoothing kernel. Then
$$
L_{\alpha} \f_{\epsilon} \geq (f^{1/m})_{\epsilon}\beta^n,
$$
in the potential sense and hence in the usual sense. Now, use Lemma \ref{lem: Gav}, we obtain 
$$
H_m(\f_{\epsilon})\geq (f^{1/m})_{\epsilon}^m\beta^n.
$$
Letting $\epsilon \to 0$ and noting that the Hessian operator is continuous under decreasing sequence, we get
$$
H_m(\f)\geq f\beta^n.
$$

\noindent {\bf Step 2: Assume that $0<f$ is only continuous.} Note that
$$
f=\sup \{h\in \Cc^{\infty}(\Omega),\ 0<h\leq f\}.
$$
Now, if $H_m(\f)\geq f\beta^n$ in the viscosity sense then we also have $H_m(\f)\geq h\beta^n$ in the viscosity sense provided that $f\geq h$. Thus, by Step 1, 
$$
H_m(\f)\geq h\beta^n,
$$
for every $0<h\leq f\in \Cc^{\infty}(\Omega).$ This yields 
$$
H_m(\f)\geq f\beta^n
$$
in the viscosity sense.
\medskip

\noindent {\bf Step 3:  $0\leq f$ is merely continuous.} We consider $\f_{\epsilon}=\f+\epsilon \vert z\vert^2.$ Then 
$$
H_m(\f_{\epsilon})\geq (f+\epsilon^m)\beta^n
$$
in the viscosity sense. By Step 2 we have 
$$
H_m(\f_{\epsilon})\geq (f+\epsilon^m)\beta^n
$$
in the potential sense and the result follows by letting $\epsilon$ go to $0.$
\end{proof}

\begin{theorem}\label{thm: viscosity vs potential general case}
Let $F: \Omega\times \R\to \R^+$ be a continuous function which is non-decreasing in the second variable. Let $\f$ be a bounded u.s.c. function in $\Omega.$ Then the inequality 
\begin{equation}\label{eq: viscosity vs potential 2}
H_m(\f)\geq F(x,\f)\beta^n
\end{equation}
holds in the viscosity sense if and only if $\f$ is $m$-subharmonic in $\Omega$ and (\ref{eq: viscosity vs potential 2}) holds in the potential sense.
\end{theorem}
\begin{proof}
Let us prove the first implication. Assume that (\ref{eq: viscosity vs potential 2}) holds in the viscosity sense.
Consider the sup-convolution of $\f$:
\begin{equation}\label{eq: sup-convolution local}
\f^{\delta}(x):=\sup\{\f(y)-\frac{1}{\delta^2}\vert x-y\vert^2 \ / \  y\in \Omega\}, \ x\in \Omega_{\delta},
\end{equation}
where $\Omega_{\delta}:=\{x\in \Omega \ /\  d(x,\pO)>A\delta\},$ and the positive constant $A$ is chosen so that $A^2>\text{osc}_{\Omega}\f.$ Then $\f^{\delta}\downarrow \f$ and  as in \cite{Ish89} (see also \cite{EGZ11}) it can be shown  that
\begin{equation}\label{eq: viscosity vs potential 3}
H_m(\f^{\delta})\geq F_{\delta}(x,\f^{\delta})\beta^n, \ \text{in} \ \Omega_{\delta},
\end{equation}
in the viscosity sense, where $F_{\delta}(x,t)=\inf_{\vert y-x\vert\leq A\delta} F(y,t).$

It follows from Proposition \ref{prop: viscosity subsolution vs potential subsolution continuous rhs} that (\ref{eq: viscosity vs potential 3}) holds in the potential sense and the result follows by letting $\delta$ go to $0.$ 

Let us prove the other implication. Suppose that $\f$ satisfies (\ref{eq: viscosity vs potential 2}) in the potential sense.  As in \cite{EGZ11} it can be shown  that 
\begin{equation}\label{eq: viscosity vs potential 4}
H_m(\f^{\delta})\geq F_{\delta}(x,\f^{\delta})\beta^n,
\end{equation}
in the potential sense. Now,  applying Proposition
 \ref{prop: viscosity subsolution vs potential subsolution
  continuous rhs} to $\f^{\delta}$ we see that 
  (\ref{eq: viscosity vs potential 4}) holds in the 
  viscosity sense. It suffices to let $\delta\to 0.$
\end{proof}

\section{Local comparison principle}
In this section we follow \cite{YW10} (see also [CC95]) to prove a viscosity comparison principle for equation (\ref{eq: heq 1}). 
\begin{definition}\label{def: semi convex}
A function $u:\Omega \rightarrow \R$ is called semiconcave (resp. semiconvex) if there exists $K>0$ (resp. $K<0$) such that  for every $z_0\in \Omega$ there exists a quadratic polynomial $P=K\vert z\vert^2+l$, where  $l$ is an affine function, which touches $u$ from above (resp. below) at $z_0.$
\end{definition}

\begin{definition}\label{def: T2}
A function $u:\Omega\rightarrow \R$ is called punctually second order differentiable  at $z_0\in \Omega$ if there exists a quadratic polynomial $q$ such that 
$$
u(z)=q(z)+o(\vert z-z_0\vert^2) \ \text{as} \ z \to z_0.
$$
Note that such a $q$ is unique if it exists. We thus define $dd^c u (z_0), D^2 u(z_0)$ to be $dd^c q(z_0), D^2q(z_0).$
\end{definition}
The following result is  a theorem of Alexandroff-Buselman-Feller (see \cite[Theorem 1, Section 6.4]{EG92}, or \cite[Theorem 1, Section 1.2]{Kr87}, or \cite[Appendix 2]{Kr87}). 

\begin{theorem}\label{thm: ABP thm}
Every continuous semiconvex (or semiconcave) function  is punctually second order  differentiable almost everywhere.
\end{theorem}

\begin{theorem}[Local comparison principle]\label{thm: viscosity comparison principle}
Let $F$ be a continuous function which is non-decreasing in the second variable. Let $u$ be a bounded viscosity subsolution and $v$ be a bounded viscosity  supersolution of 
$$
-H_m(\f)+F(x,\f)\beta^n=0.
$$ 
If $u\leq v$ on $\pO$  then $u\leq v$ on $\Omega.$
\end{theorem}
\begin{proof}
By considering $u-\epsilon$, $\epsilon>0$ and then letting $\epsilon\to 0$ noting that  $F$ is non-decreasing in the second variable, we can assume that $u<v$ near the boundary of $\Omega.$ Assume by contradiction that there exists $x_0\in \Omega$ such that 
$$
u(x_0)-v(x_0)=a>0.
$$
Let $u^{\epsilon}, v_{\epsilon}$ be the sup-convolution and inf-convolution (which is defined similarly as in (\ref{eq: sup-convolution local})).  They are semiconvex and semiconcave functions respectively. By Dini's Lemma $w_{\epsilon}:=v_{\epsilon}-u_{\epsilon}\geq 0$ near the boundary $\pO$ for $\epsilon>0$ small enough. Thus, we can fix some open  subset $U\Subset \Omega$ such that $w_{\epsilon}\geq 0$ on $\Omega\setminus U.$ 

Fix $\epsilon>0$ small enough. Denote by $E_{\epsilon}$ the set of all points in $U$ where $w_{\epsilon},u^{\epsilon},v_{\epsilon}$ are punctually second order differentiable. Then by Theorem \ref{thm: ABP thm}, the Lebesgue measure of $U\setminus E_{\epsilon}$ is $0.$ Fix some $r>0$ such that $\Omega\subset B_r\subset B_{2r}.$ Define
$$
G_{\epsilon}(x)=\sup\{\f(x) \ /\  \f\ \text{is convex in}\ B_{2r},\ \f\leq \min (w_{\epsilon},0)\ \text{in}\ \Omega\}.
$$
Since $w_{\epsilon}\geq 0$ on $\partial U$ and $w_{\epsilon}(x_0)\leq a<0$, using Alexandroff-Bakelman-Pucci (ABP) estimate (see also \cite[Lemma 4.7]{YW10})  we can find $x_{\epsilon}\in E_{\epsilon}$ such that 

(i) $w_{\epsilon}(x_{\epsilon})=G_{\epsilon}(x_{\epsilon})<0,$

(ii) $G_{\epsilon}$ is punctually second order differentiable at $x_{\epsilon}$ and det$_{\R}(D^2 G_{\epsilon}(x_{\epsilon}))\geq \delta,$ where $\delta>0$ depends only on $a,n$ and $diam(\Omega).$

Since $G_{\epsilon}$ is convex, we also have det$_{\C} (dd^c G_{\epsilon})(x_{\epsilon})\geq \delta^{1/2}.$ It follows from G{\aa}rding's inequality \cite{Ga59} that 
$$
(dd^c G_{\epsilon})^m\wedge \beta^{n-m}(x_{\epsilon})\geq \delta_1\beta^n,
$$
where $\delta_1$ does not depend on $\epsilon.$ On the other hand,   
$$
H_m(u^{\epsilon})(x_{\epsilon})\geq F_{\epsilon}(x_{\epsilon},u^{\epsilon}(x_{\epsilon}))\beta^n,
$$

Moreover $G_{\epsilon}+u^{\epsilon}$ touches $v_{\epsilon}$ from below at $x_{\epsilon}.$ Since $G_{\epsilon}+u^{\epsilon}$ is $m$-subharmonic and punctually second order differentiable  at $x_{\epsilon}$ it follows that 
$$
H_m(G_{\epsilon}+u^{\epsilon})(x_{\epsilon})\leq F^{\epsilon}(x_{\epsilon},v_{\epsilon}(x_{\epsilon}))\beta^n.
$$ 
Since $F$ is non-decreasing in the second variable and since $w_{\epsilon}(x_{\epsilon})<0$, the above inequality  implies that
$$
\delta_2+F_{\epsilon}(x_{\epsilon},u^{\epsilon}(x_{\epsilon}))\leq F^{\epsilon}(x_{\epsilon},u^{\epsilon}(x_{\epsilon})),
$$
where $\delta_2>0$ is another constant which does not depend on $\epsilon.$ 
Letting $\epsilon\to 0$, after a subsequence if necessary we obtain a contradiction.
\end{proof}

\section{Viscosity solutions on homogeneous compact Hermitian manifolds}
In this section we consider viscosity solutions to
\begin{equation}\label{eq: hes hom}
-(\omega+dd^c \f)^m\wedge \omega^{n-m}+F(x,\f)\omega^n=0, 
\end{equation}
where $(X,\omega)$ satisfies (H1), (H2) and (H3).

The notion of viscosity subsolutions and supersolutions are defined similarly as in the local case.  We compare viscosity and potential subsolutions in the two following theorems.
\begin{prop}\label{prop: viscosity vs potential hom}
Assume that  $\omega$ is K\"{a}hler and $\f$ is a continuous function on $X.$ Then $\f$ is $(\omega,m)$-subharmonic iff 
\begin{equation}\label{eq: viscosity vs potential hom}
(\omega+dd^c\f)^m \wedge \omega^{n-m}\geq 0
\end{equation}
in the viscosity sense.
\end{prop}
\begin{proof}
Assume that $\f$ is $(\omega,m)$-subharmonic and let $\f_{\epsilon}$ be the smooth regularizing sequence of $\f$ as in \cite{Chi12}. Then $\f_{\epsilon}$ is $(\omega,m)$-subharmonic in the viscosity sense. 

Fix $x_0\in X$, $\delta>0$ and $q$ a $\Cc^2$ function which touches $\f$ from above at $x_0.$ Let $B$ be a small closed ball where the touching appears and let $x_{\epsilon}$ be a maximum point of $\f_{\epsilon}-q-\delta \rho$ in $B.$ Here $\rho=\vert z-x_0\vert^2.$ Then due to the uniform convergence of $\f_{\epsilon}$ and Dini's Lemma we have $x_{\epsilon}\to x_0$ as $\epsilon\downarrow 0.$ Also, for small $\epsilon>0$, $q+\delta\rho+\f_{\epsilon}(x_{\epsilon})-q(x_{\epsilon})$ touches $\f_{\epsilon}$ from above at $x_{\epsilon}$. This implies that
$$
(\omega+dd^c q+\delta dd^c \rho)^m\wedge\omega^{n-m}\geq 0
$$ 
holds at $x_{\epsilon}$ which, in turn, implies one implication by letting $\epsilon\downarrow 0$ and $\delta\downarrow 0.$

Let us prove the other implication. Assume that $\f$ satisfies (\ref{eq: viscosity vs potential hom}) in the viscosity sense. Fix $\alpha=\alpha_1\wedge...\wedge\alpha_{m-1},$ where $\alpha_i$ are smooth $(\omega,m)$-positive closed (1,1)-forms. By G{\aa}rding's inequality we see that
$$
(\omega+dd^c \f)\wedge \alpha\wedge\omega^{n-m}\geq 0
$$ 
in the viscosity sense. Thanks to \cite[Corollary 7.20]{HKM93} the same arguments as in \cite[page 147]{Hor94} show that the above inequality also holds in the sense of currents. Thus, $\f$ is $(\omega,m)$-subharmonic.  
\end{proof}
\begin{theorem}\label{thm: viscosity vs potential hom}
Assume that $\omega$ is K\"{a}hler, $F$ is continuous on $X\times \R$ and increasing in the second variable, and $\f\in \Cc(X).$ Then $\f$ is $(\omega,m)$-subharmonic and satisfies 
\begin{equation}\label{eq: viscosity vs potential hom 2}
(\omega+dd^c\f)^m\wedge\omega^{n-m}\geq F(x,\f)\omega^n
\end{equation} 
in the potential sense  if and only if  the above inequality holds in the viscosity sense.
\end{theorem}
\begin{proof}
Set 
$$
f(x)=F(x,\f(x)), \ x\in X.
$$
\medskip

Assume that $\f$ satisfies (\ref{eq: viscosity vs potential hom 2}) in the potential sense. 
Let $x_0\in X$ and $q\in \Cc^2(U)$ which touches $\f$ from above at $x_0$ in $U$, 
a small neighborhood of $x_0.$ Suppose by contradiction that $$
(\omega+dd^cq)^m\wedge\omega^{n-m}< f\omega^n
$$
holds at $x_0.$ Then for $\epsilon$ small enough we have
$$
(\omega+dd^cq_{\epsilon})^m\wedge\omega^{n-m}< f\omega^n
$$
in a small ball $B$ containing $x_0.$ Here $q_{\epsilon}=q+\epsilon\vert x-x_0\vert^2$ 
defined in a local chart near $x_0.$ Since $q$ touches $\f$ from above at $x_0$ in $B$, 
we can find $\delta>0$ small enough such that $q_{\epsilon}-\delta\geq \f$ on $\partial B$. 
But $q_{\epsilon}(x_0)-\delta< \f(x_0)$ which contradicts the potential comparison principle. 

Now, we prove the other implication. Assume that $\f$ satisfies (\ref{eq: viscosity vs potential hom 2}) in the viscosity sense. Then from Proposition \ref{prop: viscosity vs potential hom} we see that $\f$ is $(\omega,m)$-subharmonic. 

We consider two cases.

\medskip

\noindent {\bf Case 1:  $F$ does not depend on the second variable.}
\medskip
We denote  $f (x) = F (x,0)$ for $x \in X$.

We first treat the case when $f>0$ . Fix $\tilde{f}$ a smooth function such that $0<\tilde{f}\leq f.$ Then $\f$ satisfies 
$$
(\omega+dd^c\f)^m\wedge\omega^{n-m}\geq \tilde{f}\omega^n
$$
in the viscosity sense.

Fix $\alpha=\alpha_1\wedge...\wedge\alpha_{m-1},$ where $\alpha_i$ are smooth
 $(\omega,m)$-positive closed (1,1)-forms and set 
$$
\alpha_j^m\wedge\omega^{n-m}=h_j\omega^n,\ \ j=1...m-1.
$$
From G{\aa}rding's inequality we see that
\begin{equation}\label{eq: Garding 1}
(\omega+dd^c\f)\wedge\alpha\wedge\omega^{n-m}\geq 
h_1^{1/m}...h_{m-1}^{1/m}\tilde{f}^{1/m}\omega^n,
\end{equation}
in the viscosity sense. As in the proof of Proposition \ref{prop: viscosity vs potential hom}
 it also holds in the potential sense. Let $\f_{\epsilon}$ be the smooth regularization of $\f$ 
 constructed in \cite{Chi12}. We claim that 
$$
(\omega+dd^c\f_{\epsilon})\wedge\alpha\wedge\omega^{n-m}\geq h_1^{1/m}...h_{m-1}^{1/m}(\tilde{f}^{1/m})_{\epsilon}\omega^n,
$$
in the usual sense pointwise on $X$.  Indeed, recall the definition of $\f_{\epsilon}$:
$$
\f_{\epsilon}(x)=\int_K \Lct_g \f(x)\chi_{\epsilon}(g)dg,
$$
where by $\mathcal{L}_g$ we denote the left translation by $g$, i.e $\mathcal{L}_g(x)=g.x, \ \forall x\in X.$
We compute
\begin{eqnarray*}
(\omega+dd^c\f_{\epsilon})\wedge\alpha\wedge\omega^{n-m}&=&\int_K \Lct_g\Big((\omega+dd^c \f)\wedge\Lct_{g^{-1}}\alpha\wedge\omega^{n-m}\Big)\chi_{\epsilon}(g)dg\\
(\text{ By }\ (\ref{eq: Garding 1}))\ \ &\geq& \int_K \Lct_g\Big(\Lct_{g^{-1}}(h_1^{1/m}...h_{m-1}^{1/m})\tilde{f}^{1/m}\omega^n\Big)\chi_{\epsilon}(g)dg\\
&=&h_1^{1/m}...h_{m-1}^{1/m}(\tilde{f}^{1/m})_{\epsilon}\omega^n.
\end{eqnarray*}
Thus, the claim is proved. By choosing $\alpha_j=(\omega+dd^c\f_{\epsilon}),\ j=1,...,m-1$ it follows that  
$$
(\omega+dd^c\f_{\epsilon})^m\wedge\omega^{n-m}\geq ((\tilde{f}^{1/m})_{\epsilon})^m\omega^n.
$$
By letting $\epsilon\downarrow 0$ we get 
$$
(\omega+dd^c\f)^m\wedge\omega^{n-m}\geq \tilde{f}\omega^n,
$$
in the potential sense.  Since $\tilde{f}$ was chosen arbitrarily, we deduce that
$$
(\omega+dd^c\f)^m\wedge\omega^{n-m}\geq f\omega^n,
$$
in the viscosity sense.

If $0\leq f$ is continuous we consider $\f_t:=(1-t)\f+t\psi$ where $\psi$ is a smooth strictly $(\omega,m)$-subharmonic function and $0<t<1.$ Then for each fixed $t\in (0,1)$,  $\f_t$ satisfies 
\begin{equation}\label{eq: viscosity potentiel 1}
(\omega+dd^c\f_t)^m\wedge\omega^{n-m}\geq f_t\omega^n,
\end{equation}
in the viscosity sense with  $f_t$ continuous and strictly positive:
$$
f_t=(1-t)^m f +t^m \frac{(\omega+dd^c\psi)^m\wedge\omega^{n-m}}{\omega^n}.
$$
We then can apply what we have done above to infer that $\f_t$ verifies  (\ref{eq: viscosity potentiel 1}) in the potential sense. It suffices now to let $t\downarrow 0.$

\medskip
\noindent{\bf Case 2: $F$ depends on the second variable.}  Since $\f$ is continuous, the function $f: X\rightarrow \R$, $f(x)= F(x,\f(x))$ is continuous. We can apply Case 1 to complete the proof.
\end{proof}

\subsection{Global Comparison Principle}
Let $u, v$ be  bounded  viscosity subsolution and  supersolution of (\ref{eq: hes hom}).  Construct a distance $d$ on $K$ such that $d^2: K\times K\rightarrow \R^+$ is smooth.  Consider the sup-convolution and inf-convolution as follows
\begin{equation}\label{eq: sup-convolution}
u^{\epsilon}(x):=\sup\Big\{u(g.x)-\frac{1}{\epsilon^2}d^2(g,e)\ /\ \ g\in K\Big\},
\end{equation}
and 
\begin{equation}\label{eq: inf-convolution}
v_\epsilon(x):=\inf\Big\{v(g.x)+\frac{1}{\epsilon^2}d^2(g,e)\ /\ \ g\in K\Big\}.
\end{equation}
\begin{lemma}\label{lem: semi convex}
Fix $x_0\in X$ and consider local coordinates $z: \Omega\to B(0,2),$ where $\Omega$ 
is a small open neighborhood of $x_0$ and $B$ is the ball of radius $2$ in $\C^n.$  
Then $u^{\epsilon}, v_{\epsilon}$ read in this local chart as semiconvex and semiconcave
 functions. In particular, they are punctually second order differentiable almost everywhere in $B(0,1).$
\end{lemma}
\begin{proof}
We only need to prove the result for $u^{\epsilon}$ since for $v_{\epsilon}$ it follows similarly. 
Consider a smooth section $s: \Omega \rightarrow K$ such that
 $\pi\circ s(x)=x, \forall x\in \Omega,$ where $\pi$ is the projection of $K$ onto $X.$
 
 For simplicity we identify a point in $\Omega$ with its image in $B(0,2).$ 
 
 Put 
$$
 \rho(x)=u^{\epsilon}+C \vert x\vert^2,
$$
where $C>0$ is a big constant to be specified later.
 
We claim that for any $x\in B(0,1)$ there exists $\delta>0$ such that
$$
\rho(x+h)+\rho(x-h)\geq 2\rho(x), \ \forall h\in \C^{n}, \vert h\vert\leq \delta.
$$

It is classical that this property implies the convexity of $\rho.$ Let us prove the claim.  
Let $x_0\in B(0,1)$ and $y_0=g_0.x_0$ be such that 
\begin{equation}\label{eq: repair 1}
u^{\epsilon}(x_0)=u(y_0)-\frac{1}{\epsilon^2}d^2(g_0,e).
\end{equation}
By considering $\epsilon>0$ small enough we can assume that $y_0\in B(0,3/2).$ For  $h\in \C^{n}$  small enough such that $x_0+h, x_0-h\in B(0,1)$, set
$$
\theta(h)= g_0.s(x_0).s(x_0+h)^{-1}.
$$
Then it is easy to see that $\theta(h).(x_0+h)=y_0.$ By definition of $u^{\epsilon}$ we thus get
\begin{equation}\label{eq: repair 2}
u^{\epsilon}(x_0+h)\geq u(y_0)-\frac{1}{\epsilon^2}d^2(\theta(h),e),
\end{equation}
and 
\begin{equation}\label{eq: repair 3}
u^{\epsilon}(x_0-h)\geq u(y_0)-\frac{1}{\epsilon^2}d^2(\theta(-h),e).
\end{equation}
From (\ref{eq: repair 1}), (\ref{eq: repair 2}) and (\ref{eq: repair 3}) we obtain
$$
u^{\epsilon}(x_0+h)+u^{\epsilon}(x_0-h)-2u^{\epsilon}(x_0)\geq -
\frac{1}{\epsilon^2} \Big(d^2(\theta(h),e)+d^2(\theta(-h),e)-2d^2(\theta(0),e)\Big).
$$
Since $s$ is smooth and $K$ is compact we can choose $C>0$ big enough (does not depend on $x_0$) such that  
$$
u^{\epsilon}(x_0+h)+u^{\epsilon}(x_0-h)-2u^{\epsilon}(x_0)\geq -2C\vert h\vert ^2,
$$
for $h\in \C^n$ small enough. This proves the  claim. The last statement follows from Alexandroff-Buselman-Feller's theorem (Theorem \ref{thm: ABP thm}).
\end{proof}
\begin{lemma}\label{lem: comparison}
$u^{\epsilon}$ is a viscosity subsolution of 
\begin{equation}\label{eq: comparison 1}
-(\omega+dd^cu)^m\wedge \omega^{n-m}+F_{\epsilon}(x,u)\omega^n=0,
\end{equation}
where 
$$
F_{\epsilon}(x,t):=\inf\Big\{F(g.x,t)\ / \ g\in K, d(g,e)\leq \sqrt{\text{osc}(u)} \epsilon\Big\}.
$$
Similarly,  $v_{\epsilon}$ is a viscosity supersolution of 
\begin{equation}\label{eq: comparison 2}
-(\omega+dd^cu)^m\wedge \omega^{n-m}+F^{\epsilon}(x,u)\omega^n=0,
\end{equation}
where
$$ 
F^{\epsilon}(x,t):=\sup\Big\{F(g.x,t)\ / \ g\in K, d(g,e)\leq \sqrt{\text{osc}(v)}\epsilon\Big\}.
$$ 
\end{lemma}
\begin{proof}
We only need to prove the first assertion since the second one follows similarly. Let $q$ be a function of class $\Cc^2$ in a neighborhood of $x_0\in X$ that touches $u^{\epsilon}$ from above at $x_0.$ Let $g_0\in K$ be such that 
$$
u^{\epsilon}(x_0)=u(g_0.x_0)-\frac{1}{\epsilon^2}d^2(g_0,e).
$$
Consider the function $Q$ defined by 
$$
Q(x):=q(g_0^{-1}x)+\frac{1}{\epsilon^2}d^2(g_0,e).
$$
Then $Q$ touches $u$ from above at $g_0.x_0.$ Since $u$ is  a subsolution of (\ref{eq: hes hom}), we have
$$
(\omega+dd^c Q)^m\wedge \omega^{n-m}\geq F(x,Q)\omega^n, \ \text{at} \ g_0.x_0.
$$
Since $\mathcal{L}^*_{g_0}\omega=\omega$ we get
$$
(\omega+dd^c q)^m\wedge \omega^{n-m}\geq F(g_0.x_0,q(x_0))\omega^n, \ \text{at} \ x_0.
$$
From the definition of $u^{\epsilon}$ we know that $u^{\epsilon}(x_0)=u(g_0.x_0)-\frac{1}{\epsilon^2}d^2(g_0,e)\geq u(x_0).$ 
Thus $d(g_0,e)\leq \epsilon\sqrt{\text{osc}(u)}$ and the result follows.
\end{proof}

Now, we prove a viscosity comparison principle on homogeneous manifolds. The fact that the metric $\omega$ is invariant under group actions allows us to follow the proof of Theorem \ref{thm: viscosity comparison principle} in this global context.
\begin{theorem}\label{thm: viscosity comparison principle hom}
Assume that $u, v$ are bounded viscosity subsolution and supersolution of 
$$
-(\omega+dd^c\f)^m\wedge\omega^{n-m}+ F(x,\f)\omega^n=0,
$$
where $0\leq F(x,t)$ is a continuous function which is increasing in the second variable. Then we have $u\leq v$ on $X.$
\end{theorem}
\begin{proof}
We consider the sup-convolution and inf-convolution of $u,v$ as in (\ref{eq: sup-convolution}) and (\ref{eq: inf-convolution}).
These functions read in local coordinates as semiconvex and semiconcave functions which are punctually second order differentiable almost everywhere. For each $\epsilon>0$ let $x_\epsilon$ be a maximum point of $u^\epsilon-v_\epsilon$ on $X.$ 

We first treat the case when  $u^{\epsilon}, v_\epsilon$ are punctually second order differentiable at $x_\epsilon.$ In this case, by the classical maximum principle we have
$$
dd^c u^{\epsilon}\leq dd^c v_\epsilon \ \text{at}\ x_\epsilon.
$$
The form  $(\omega+ dd^c u^{\epsilon})$ is $(\omega,m)$-positive at $x_{\epsilon}$. Thus,  
$$
(\omega+dd^cu^{\epsilon})^m\wedge \omega^{n-m}\leq (\omega+dd^cv_{\epsilon})^m\wedge \omega^{n-m} \ \text{at}\ x_\epsilon,
$$
and hence Lemma  \ref{lem: comparison} yields 
\begin{equation}\label{eq: comparison 3}
F_{\epsilon} (x_\epsilon, u^{\epsilon})\leq F^{\epsilon} (x_\epsilon, v_{\epsilon}).
\end{equation}
We can assume that $x_\epsilon\to x_0\in X.$ By extracting a subsequence (twice), there exists a sequence $\epsilon_j\downarrow 0$ such that 
$$
 F_{\epsilon_j} (x_{\epsilon_j}, u^{\epsilon_j}(x_{\epsilon_j}))\ \text{and}\  F^{\epsilon_j} (x_{\epsilon_j}, v_{\epsilon_j}(x_{\epsilon_j}))
$$
converge when $j\to+\infty.$ We thus deduce from (\ref{eq: comparison 3}) that   
$$
F(x_0,\liminf_{j} u^{\epsilon_j}(x_{\epsilon_j}))\leq F(x_0,\limsup_{j} v_{\epsilon_j}(x_{\epsilon_j})).
$$
Since $F$ is increasing in the second variable the latter implies that
\begin{equation}\label{eq: comparison 4}
\liminf u^{\epsilon_j}(x_{\epsilon_j})\leq \limsup v_{\epsilon_j}(x_{\epsilon_j}).
\end{equation}
Since $u^{\epsilon}\downarrow u$ and $v_{\epsilon}\uparrow v$ we have
$$
\sup_X (u-v)\leq \sup_X (u^{\epsilon_j}-v_{\epsilon_j})=u^{\epsilon_j}(x_{\epsilon_j})-v_{\epsilon_j}(x_{\epsilon_j}).
$$
Then (\ref{eq: comparison 4}) implies that $\sup_X(u-v)\leq 0.$
\medskip

Now, if $u^{\epsilon}, v_\epsilon$ are not punctually second order differentiable at $x_\epsilon$ for fixed $\epsilon,$ we proceed as in \cite{EGZ11} to prove that (\ref{eq: comparison 3}) still holds. Consider a local holomorphic chart centered at $x_\epsilon.$ For simplicity we identify a point near $x_{\epsilon}$ with its image in $\C^n.$ For each $k\in \N^*$, the semiconvex function $u^{\epsilon}-v_{\epsilon}-\frac{1}{2k}\Vert x-x_{\epsilon}\Vert^2$ attains its strict maximum at $x_{\epsilon}.$  By Jensen's lemma (\cite{Jen88}; see also \cite[Lemma A.3, page 60]{CIL92}), there exist  sequences $(p_k), (y_k)$ converging to $0$ and $x_{\epsilon}$ respectively such that the functions $u^{\epsilon}, v_\epsilon$ are punctually second order differentiable at $y_k$ and the function 
$$
u^{\epsilon}-v_{\epsilon}-\frac{1}{2k}\Vert x-x_{\epsilon}\Vert^2-\langle p_k,x\rangle
$$
attains its local maximum at $y_k.$  We thus get 
$$
dd^c u^{\epsilon}\leq dd^c v_\epsilon+O(1/k)\omega \ \text{at}\ y_k.
$$
Since $v_\epsilon$ is semi-concave, and $u^{\epsilon}$ is $(\omega,m)$-subharmonic we get
$$
(\omega+dd^c u^{\epsilon})^m\wedge \omega^{n-m}\leq (\omega+ dd^c v_\epsilon)^m\wedge\omega^{n-m}+O(1/k)\omega^n \ \text{at}\ y_k.
$$
This together with (\ref{eq: comparison 1}) and (\ref{eq: comparison 2}) yield
$$
F_{\epsilon} (y_k, u^{\epsilon}(y_k))\leq F^{\epsilon} (y_k, v_{\epsilon}(y_k))+O(1/k).
$$
Now, let $k\to +\infty$ we obtain (\ref{eq: comparison 3}) which completes the proof.
\end{proof} 


\section{Proof of the main results}
\subsection{Proof of Theorem A}
Let $\Fc$ denote the family of all subsolutions $w$ of (\ref{eq: heq 1})  such that $u\leq w\leq v.$ It is not empty thanks to the local comparison principle.  We set 
$$
\f:=\sup\{w: w\in \Fc\}.
$$ 
By Choquet's lemma $\f^*=(\limsup w_j)^*$ where $w_j$ is a sequence in $\Fc.$ It follows from Lemma \ref{lem: decreasing sequence of viscosity subsolution} that $\f^*$ is a subsolution of (\ref{eq: heq 1}). 

We claim that $\f_*$ is a supersolution of (\ref{eq: heq 1}). Indeed, assume that $\f_*$ is not a supersolution of (\ref{eq: heq 1}). Then there exist $x_0\in \Omega$ and $q\in \Cc^2(\{x_0\})$ such that $q$ touches $\f_*$ from below at $x_0$ but 
$$
H_m(q)(x_0)>F(x_0,q(x_0))\beta^n.
$$
By the continuity of $F$, we can find  $r>0$ small enough  such that  $q\leq \f_*$ in $B(x_0,r)$ and
$$
H_m(q)(x)>F(x,q(x))\beta^n,\ \ \forall x\in B=B(x_0,r).
$$
We then choose $0<\epsilon$ small enough and $0<\delta<<\epsilon$ so that the function 
$Q=Q_{\epsilon,\delta}:=q+\delta -\epsilon \vert x-x_0\vert^2$  satisfies
$$
H_m(Q)(x)>F(x,Q(x))\beta^n, \ \forall x\in B.
$$
Define $\phi$ to be $\f$ outside $B$ and $\phi=\max(\f,Q)$ in $B.$ Since $Q<\f$ near $\partial B$, we see that $\phi$ is upper semi continuous and it is a subsolution of (\ref{eq: heq 1}) in $\Omega.$ 

Let $(x_j)$ be a sequence in $B$ converging to $x_0$ such that $\f(x_j)\to \f_*(x_0).$ Then $Q(x_j)-\f(x_j)\to Q(x_0)-\f_*(x_0)=\delta>0.$ Thus $\phi\not \equiv \f,$ which contradicts the maximality of $\f.$
\medskip

From the above steps we know that $\f_*$ is a supersolution and $\f^*$ is a subsolution. We also have $g=u_*\leq \f_*\leq \f^*\leq v^*=g$ on $\pO.$ Thus by the viscosity comparison principle $\f=\f_*=\f^*$ is a continuous viscosity solution of (\ref{eq: hessian}) with boundary value $g.$

It remains to prove that $\f$ is also a potential solution of (\ref{eq: hessian}).
From Theorem \ref{thm: viscosity vs potential general case} we know that 
$$
H_m(\f)\geq F(x,\f)\beta^n
$$
in the potential sense. Let  $B=B(x_0,r)\subset \Omega$ is a small ball in $\Omega.$  Thanks to Dinew and Kolodziej \cite[Theorem 2.10]{DK11} we can solve the Dirichlet problem to find $\p\in \Pm(B)\cap \Cc(\bar{B})$ with boundary value $\f$ such that 
$$
H_m(\p)=F(x,\f)\beta^n,\  \text{in}\ B.
$$ 
By the potential comparison principle we have $\f\leq \p$ in $\bar{B}$. Define $\tilde{\p}$ to be $\p$ in $B$ and $\f$ in  $\Omega\setminus B$. Set $G(x)=F(x,\f(x)), x\in \Omega.$ It is easy to see that $\tilde{\p}$ is a viscosity solution of 
$$
-(dd^c u)^m\wedge\beta^{n-m}+G\beta^n=0.
$$
By the viscosity comparison principle we deduce that $\tilde{\p}\leq \f$ in $\Omega$ which implies that $\f=\p$ in $B.$ The proof is thus complete.

\subsection{Proof of Theorem B}
Let $\f$ be the unique viscosity solution obtained from Theorem A. Since $u, v$ are $\gamma$-H\"{o}lder continuous in $\bar{\Omega}$ and $F$ satisfies (\ref{eq: V3}), we can find a constant $C>0$ such that 
$$
\sup_{x,y\in \bar{\Omega}}\Big(\vert u(x)-u(y)\vert+ \vert v(x)-v(y)\vert\Big)\leq C\vert x-y\vert^{\gamma},
$$
and 
$$
 \sup_{\vert t\vert \leq M}\sup_{x,y\in \bar{\Omega}}\vert F^{1/m}(x,t)-F^{1/m}(y,t)\vert\leq C\vert x-y\vert^{\gamma},
$$
where $M>0$ is such that  $\vert \f\vert \leq M,$ on $\bar{\Omega}.$ 

Fix $R>0$ such that $\Omega\subset B(0,R).$ Define  $\p:\bar{\Omega}\rightarrow \R$ by
$$
\p(x):=\sup_{y\in \bar{\Omega}}\Big\{\f(y)+C\vert x-y\vert^{\gamma}(\vert x\vert^2-R^2-1)\Big\}.
$$

\noindent {\bf Step 1: Prove that $\p$ is $\gamma$-H\"{o}lder continuous.}
Fix $x_1,x_2\in \bar{\Omega}$, and $y_1,y_2$ corresponding maximum points in $\bar{\Omega}$ as in the definition of $\p.$ We obtain
\begin{multline*}
\p(x_1)-\p(x_2)\geq  C\vert x_1-y_2\vert^{\gamma}(\vert x_1\vert^2-R^2-1) -C\vert x_2-y_2\vert^{\gamma}(\vert x_2\vert^2-R^2-1)\\
=C(\vert x_1\vert^2-R^2-1)(\vert x_1-y_2\vert^{\gamma}-\vert x_2-y_2\vert^{\gamma}) +C\vert x_2-y_2\vert^{\gamma}(\vert x_1\vert^2-\vert x_2\vert^2)\\
\geq C(\vert x_1\vert^2-R^2-1)\vert x_1-x_2\vert^{\gamma} +C\vert x_2-y_2\vert^{\gamma}(\vert x_1\vert^2-\vert x_2\vert^2)\geq -C'\vert x_1-x_2\vert^{\gamma},
\end{multline*}
where $C'>0$ depends only on $C, R.$ Similarly, we have
$$
\p(x_1)-\p(x_2)\leq C'\vert x_1-x_2\vert^{\gamma}.
$$
The above inequalities show  that $\p$ is $\gamma$-H\"{o}lder continuous in $\bar{\Omega}.$
\medskip

\noindent{\bf Step 2: Prove that $\p$ is a subsolution of (\ref{eq: heq 1}).} Let $x_0\in \Omega$ and $q\in \Cc^2(\{x_0\})$ which touches $\p$ from above at $x_0.$  Let $y_0\in \bar{\Omega}$ be such that
$$
\p(x_0)=\f(y_0)+C\vert x_0-y_0\vert^{\gamma}(\vert x_0\vert^2-R^2-1).
$$
If   $y_0\in \pO$ then $\f(y_0)=u(y_0),$ hence  
\begin{eqnarray*}
0&\geq& C\vert x_0-y_0\vert^{\gamma}(\vert x_0\vert^2-R^2)=\p(x_0)-\f(y_0)+C\vert x_0-y_0\vert^{\gamma}\\
&\geq & u(x_0)-u(y_0)+C\vert x_0-y_0\vert^{\gamma}\geq 0.
\end{eqnarray*}
We thus get $\f(x_0)=\p(x_0)$ and the result follows since $\f$ is a subsolution. Let us treat the case $y_0\in \Omega.$ 
The function $Q$, defined around $y_0$ by 
$$
Q(x):=q(x+x_0-y_0)-C\vert x_0-y_0\vert^{\gamma}\Big(\vert x+x_0-y_0\vert^2-R^2-1\Big),
$$
touches $\f$ from above at $y_0.$ Since $\f$ is a subsolution of (\ref{eq: heq 1}), we have
$$
\St_m^{1/m}\Big(dd^c Q(y_0)\Big)\geq F^{1/m}(y_0,Q(y_0)).
$$
By the concavity of $\St_m^{1/m}$ we get
\begin{eqnarray*}
\St_m^{1/m}\Big(dd^c q(x_0)\Big)&\geq& F^{1/m}(y_0,Q(y_0))+C\vert x_0-y_0\vert^{\gamma}\\
&=&F^{1/m}(y_0,\f(x_0))+C\vert x_0-y_0\vert^{\gamma}\\
&\geq& F^{1/m}(x_0,\f(x_0)),
\end{eqnarray*}
which implies that $\p$ is a subsolution of (\ref{eq: heq 1}). 

It is easy to see that $\f\leq \p$ and for any $x\in \pO, y\in \bar{\Omega}$, we have
$$
\f(y)-C\vert x-y\vert^{\gamma}\leq v(y)-C\vert x-y\vert^{\gamma}\leq v(x)=g(x).
$$
This implies $\p=g$ on $\pO.$ Hence,  since $\f$ is maximal we obtain $\f=\p$ which, in turn, shows that $\f$ is $\gamma$-H\"{o}lder continuous.


\subsection{Proof of Corollary C}
Let $h$ be the harmonic function with boundary value $g;$ it is a continuous supersolution of (\ref{eq: heq 1}).  It follows from \cite{BT76} that there exists a continuous psh function $u$ with boundary value $g.$  Then for $A>>1$, the function $u+A\rho$, where $\rho$ is a defining function of $\Omega,$ is a subsolution. Thus, by Theorem A there exists a continuous viscosity solution. 

Now, assume that  $g$ is $(2\gamma)$-H\"{o}lder continuous in $\bar{\Omega}.$ Then we can choose $u$ to be  $\gamma$-H\"{o}lder continuous in $\bar{\Omega}$ thanks to \cite{BT76}. The same thing holds for $h.$  It suffices to apply Theorem B. The proof is thus complete.
\medskip

\begin{rmk}
In Corollary C it is natural to consider a strongly $m$-pseudoconvex domain (i.e. the defining function is strongly $m$-subharmonic). The existence of continuous subsolution and supersolution is obvious which yields the existence of viscosity solution. However, the  H\"{o}lder continuity is delicate. 
\end{rmk}
\subsection{Proof of Theorem D}
It follows from (\ref{eq: V4}) that $u\equiv t_0$ is a subsolution and $v\equiv t_1$ is a supersolution of (\ref{eq: hes hom}). The global comparison principle  (Theorem \ref{thm: viscosity comparison principle hom})  allows us to repeat  the proof of Theorem A to prove Theorem D.

\medskip

In the following, we give an example of compact Hermitian homogeneous manifold satisfying our conditions (H1), (H2), (H3) which is not K\"{a}hler. It  is communicated to us by Karl Oeljeklaus to whom we are indebted.
\begin{examp}\label{example: Karl}
Consider $G= SL(3,\C)$, $K=SU(3)$ and
$$
H=\left\{ \left( \begin{array}{ccc}e^{w} & z_1 & z_2 \\0 & e^{iw} & z_3 \\ 0& 0 &e^{-w-iw} \end{array} \right)\ \Big /\ w, z_1, z_2, z_3\in \C\right\}.
$$
Then $G$ is a connected complex Lie group and $H$ is a closed complex subgroup. The manifold $X=G/H$ is Hermitian. 
It is clear that $K$ acts freely and transitively on $X.$ Taking any Hermitian metric and averaging it over the Haar measure of  $K$ we obtain a Hermitian metric $\omega$ verifying (H3).  
Now we prove that $X$ is not  K\"{a}hler.  Since $K$ acts freely on $X$ we see that $X$ is simply connected. 

Consider 
$$
I=\left\{ \left( \begin{array}{ccc} \lambda_1 & z_1 & z_2 \\0 & \lambda_2 & z_3 \\ 0& 0 &(\lambda_1.\lambda_2)^{-1} \end{array} \right)\ \Big /\  z_1, z_2, z_3\in \C; \lambda_1, \lambda_2 \in \C^* \right\}.
$$
Then $H$ is a closed subgroup of $I$ and $Y=G/I$ is a rational-projective manifold.  If $X$ admits a K\"{a}hler metric then it follows from \cite{BR62} (see also \cite{BN90})  that  the Tits fibration 
$$
\pi: G/H\rightarrow G/I
$$
is holomorphically trivial and its fiber $I/H$ is a complex compact torus. This implies that $\pi_1(X)$ is non-trivial which is impossible since $X$ is simply connected. Thus  $X$  does not admit any  K\"{a}hler metric.
\end{examp}

\bigskip

\noindent {\bf Acknowledgement.} I would like to express my deep gratitude to Ahmed Zeriahi and Vincent Guedj for inspirational discussions and constant encouragements.  
I also would like to thank  Philippe Eyssidieux and 
 Dan Popovici for valuable discussions.  I am very grateful to Nguyen Van Dong for finding an important error in the previous proof of Lemma \ref{lem: semi convex}. I am indebted to   Karl Oeljeklaus for communicating an 
 important example and references for it and also for very useful discussions.
 Finally, I would like to thank the referee for useful comments and corrections which improve the presentation of this paper.

 \bigskip

\noindent LU Hoang Chinh\\
Universit\'e Paul Sabatier\\
Institut de Math\'ematiques de Toulouse\\
118 Route de Narbonne \\
31062 Toulouse\\
lu@math.univ-toulouse.fr. 
\end{document}